\documentclass{amsart}
\usepackage{amssymb,enumerate}
\newtheorem{lemma}{Lemma}[section]
\newtheorem{theorem}[lemma]{Theorem}

\newtheorem{proposition}[lemma]{Proposition}

\newtheorem{question}[lemma]{Question}

\newcommand{\ad}{{\rm ad}}

\newcommand{\F}{{\mathbb F}}

\newcommand{\Z}{{\mathbb Z}}

\newcommand{\Aut}{{\rm Aut}}

\newcommand\chr{\mathrm{char}}

\newcommand{\al}{\alpha}
\newcommand{\bt}{\beta}
\newcommand{\dl}{\delta}
\newcommand{\lm}{\lambda}
\newcommand{\sg}{\sigma}
\newcommand{\om}{\omega}
\newcommand{\Gm}{\Gamma}
\newcommand{\M}{\mathcal M}

\newcommand{\Mab}{\mathcal M(\alpha, \beta)}

\title{An infinite-dimensional $2$-generated primitive axial algebra of Monster type}

\author{Clara Franchi, Mario Mainardis and Sergey Shpectorov}

\address{Dipartimento di Matematica e Fisica,
Universit\`a Cattolica del Sacro Cuore,
Via Musei 41,
I-25121 Brescia, Italy}
\email{clara.franchi@unicatt.it}

\address{Dipartimento di Scienze Matematiche, Informatiche e Fisiche, 
Universit\`a degli Studi di Udine, via delle Scienze 206,
I-33100 Udine, Italy}
\email{mario.mainardis@uniud.it}

\address{School of Mathematics,
University of Birmingham, 
Watson Building, Edgbaston,
Birmingham, B15 2TT, UK}
\email{s.shpectorov@bham.ac.hk}

\begin{document}
\maketitle

\begin{abstract}
Rehren proved in \cite{RT,R} that a primitive 2-generated axial algebra of Monster type 
$(\al,\bt)$, over a field of characteristic other than $2$, has dimension at most eight if $\al\notin\{2\bt,4\bt\}$. In this note we construct 
an infinite-dimensional $2$-generated primitive axial algebra of Monster type $(2,\frac{1}{2})$ 
over an arbitrary field $\F$ with $\chr(\F)\neq 2,3$. This shows that the second special case, 
$\al=4\bt$, is a true exception to Rehren's bound.
\end{abstract}

\section{Introduction}
 Let  $\F$ be a field and let $\mathcal S$ be a finite subset of $\F$ with $1\in\mathcal S$. 
A {\it fusion law} on $\mathcal S$ is a map 
$$\star\colon \mathcal S \times \mathcal S \to 2^{\mathcal S}.$$
An {\it axial algebra}  over $\F$ with {\it spectrum} $\mathcal S$ and fusion law $\star$ is a 
commutative non-associative $\F$-algebra $V$ generated by a set $\mathcal A$ of nonzero 
idempotents (called {\it axes}) such that, for each $a\in {\mathcal A}$, 
\begin{enumerate}
\item[(Ax1)] $ad(a):v\mapsto av$ is a semisimple endomorphism of $V$ with spectrum contained in 
$\mathcal S$;
\item[(Ax2)] for every $\lm,\mu\in\mathcal S$,  the product of a $\lm$-eigenvector and a 
$\mu$-eigenvector of $\ad_a$ is the sum of $\dl$-eigenvectors, for $\dl\in\lm\star\mu$.
\end{enumerate} 
Furthermore, $V$ is called {\it primitive} if 
\begin{enumerate}
\item[(Ax3)] $V_1=\langle a \rangle$.
\end{enumerate}
An axial algebra over $\F$ is said to be {\it of Monster type} $(\al,\bt)$ if it satisfies 
the fusion law $\Mab$ given in Table~\ref{Ising}, with $\al,\bt\in\F\setminus\{0,1\}$, with 
$\al\neq\bt$.
\begin{table}
$$ 
\begin{array}{|c||c|c|c|c|}
\hline
\star & 1 & 0 & \al & \bt\\
\hline
\hline
1 & 1 & & \al & \bt\\
\hline
0 & & 0 & \al & \bt\\
\hline
\al	& \al & \al & 1,0 & \bt\\
\hline
\bt & \bt & \bt & \bt & 1,0,\al\\
\hline
\end{array}
$$
\bigskip
\caption{Fusion law $\M(\al,\bt)$}\label{Ising}
\end{table} 
    
Axial algebras were introduced by Hall, Rehren and Shpectorov~\cite{HRS,Jordan}  in 
order to axiomatise some key features of certain classes of algebras, such as the weight-2 
components of OZ-type vertex operator algebras, Jordan algebras, Matsuo algebras and Majorana 
algebras (see~\cite{Mb} and the introductions of \cite{HRS}, \cite{R} and \cite{3A}). In this paper we shall assume that the underlying field $\F$ has characteristic different from $2$. This case is of particular interest for finite  group theorists, since most of the finite simple groups can be faithfully represented as groups generated by certain involutory automorphisms, called Miyamoto 
involutions \cite{Miya02}, of such algebras. 
In particular, the 
Griess algebra (see~\cite{Gri}) is a real axial algebra of Monster type $(\frac{1}{4},\frac{1}{32})$ and the Myiamoto involutions of this algebra (also called Majorana involutions) are precisely the involutions of type $2A$ in the Monster, i.e. those whose centraliser in the Monster is the double cover of the Baby Monster. 
    
The classification of $2$-generated axial algebras has a fundamental r\^ole in the development of 
the theory of axial algebras. In a pioneering work \cite{N96}, Norton classified subalgebras of 
the Griess algebra generated by two axes. Norton showed that there are exactly nine isomorphism 
classes of such subalgebras, corresponding to the nine conjugacy classes of dihedral subgroups 
of the Monster generated by two involutions of type $2A$. These algebras have been proven to be, up to 
isomorphisms, the only $2$-generated primitive real axial algebras of Monster type 
$(\frac{1}{4},\frac{1}{32})$ and are now known as Norton-Sakuma algebras (\cite{S}, 
\cite{IPSS10}, \cite{HRS}, \cite{FMS1}). In the minimal non-associative case of axial algebras 
of Jordan type $\eta$, the classification has been obtained by Hall, Rehren and Shpectorov in 
\cite{Jordan}. Note that axial algebras of Jordan type $\eta$ are also axial algebras of Monster 
type $(\al,\bt)$ when $\eta\in\{\al,\bt\}$. In \cite{RT} and \cite{R}, Rehren started a 
systematic study of axial algebras of Monster type $(\al,\bt)$. In particular, he showed that, 
when $\al\not\in\{2\bt,4\bt\}$ every $2$-generated primitive axial algebra of Monster type 
$(\al,\bt)$ has dimension at most $8$. 

This note is part of a project of the authors aimed at classifying all $2$-generated primitive 
axial algebras of Monster type. In particular, in \cite{FMS1} and \cite{FMS2}, the authors 
extend Rehren's result showing that, if $(\al,\bt)\neq(2,\frac{1}{2})$, then any symmetric 
$2$-generated primitive axial algebra of Monster type $(\al,\bt)$ has dimension at most $8$ 
({\it symmetric} means that the map that swaps the generating axes extends to an algebra 
automorphism). Here we prove that the case $(\al,\bt)=(2,\frac{1}{2})$ is indeed an exception.  

\begin{theorem}\nonumber \label{main}
For every field $\F$ of characteristic different from $2$ and $3$, there exists an 
infinite-dimensional $2$-generated primitive axial algebra of Monster type $(2,\frac{1}{2})$
over $\F$.  
\end{theorem}

%%%%%%%%%%%%%%%%%%%%%%%%%%%%%%%%%%
%%%%%%%%%%%%%%%%%%%
\section{The algebra $HW$}

Let $\F$ be a field of characteristic other than $2$. Let $HW$ be the infinite-dimensional 
$\F$-vector space with basis $\{a_i,\sg_j\mid i\in\Z,j\in\Z_+\}$,
$$
HW:=\bigoplus_{i\in \Z}\F a_i\oplus\bigoplus_{j\in \Z_+}\F\sg_j.
$$
Set $\sg_0=0$. Define a commutative non-associative product on $HW$ by linearly extending the 
following values on the basis elements:
\begin{enumerate}
\item[(HW1)] $a_ia_j:=\frac{1}{2}(a_i+a_j)+\sg_{|i-j|}$;\\
\item[(HW2)] $a_i\sg_j:=-\frac{3}{4}a_i+\frac{3}{8}(a_{i-j}+a_{i+j})+\frac{3}{2}\sg_j$;\\
\item[(HW3)] $\sg_i\sg_j:=\frac{3}{4}(\sg_i+\sg_j)-\frac{3}{8}(\sg_{|i-j|}+\sg_{i+j})$.
\end{enumerate} 

\medskip\noindent
In particular, $a_i^2=\frac{1}{2}(a_i+a_i)+\sg_0=a_i$, so each $a_i$ is an idempotent.

We call $HW$ the \emph{highwater algebra} because it was discovered in Venice during the 
disastrous floods in November 2019. In what follows, double angular brackets denote 
algebra generation while single brackets denote linear span.

\begin{theorem}\label{HW}
If $\chr(\F)\neq 3$ then $HW=\langle\!\langle a_0,a_1\rangle\!\rangle$ is a primitive axial 
algebra of Monster type $(2,\frac{1}{2})$.
\end{theorem}

Manifestly, this result implies Theorem \ref{main}.

If $\chr(\F)=3$ then $2=\frac{1}{2}$ and so the concept of an algebra of Monster type 
$(2,\frac{1}{2})$ is not defined. However, the four-term decomposition typical for algebras of 
Monster type still exists, and so $HW$ in characteristic 3 is an example of an axial 
decomposition algebra as defined in \cite{DSV}. We also prove that $HW$ in this case is a Jordan 
algebra. Note, however, that because a lot of structure constants in $HW$ become zero in 
characteristic $3$, $HW$ is no longer generated by $a_0$ and $a_1$. In fact, every pair of 
distinct axes $a_i, a_j$ generates a $3$-dimensional subalgebra, linearly spanned by $a_i$, 
$a_j$, and $\sigma_{|i-j|}$, and  isomorphic to the algebra $Cl^{00}(\F^2,b_2)$ (see 
\cite[Theorem~(1.1)]{Jordan}).

We start with a number of observations concerning the properties of $HW$. First of all, we show 
that it is not a simple algebra. Consider the linear map $\lm:HW\to\F$ defined on the basis of 
$HW$ as follows: $\lm(a_i)=1$ for all $i\in\Z$ and $\lm(\sg_j)=0$ for all $j\in\Z_+$. 

\begin{lemma}\label{lambda}
The map $\lm$ is a homomorphism of algebras.
\end{lemma}

\begin{proof}
It suffices to show that $\lm$ is multiplicative, i.e., $\lm(uv)=\lm(u)\lm(v)$ for all $u,v\in 
HW$. Since this equality is linear in both $u$ and $v$, it suffices to check it for the elements 
of the basis. If $u=a_i$ and $v=a_j$ then $\lm(a_ia_j)=\lm(\frac{1}{2}(a_i+a_j)+\sg_{|i-j|})=
\frac{1}{2}+\frac{1}{2}=1=1\cdot 1=\lm(a_i)\lm(a_j)$. If $u=a_i$ and $v=\sg_j$ then 
$\lm(a_i\sg_j))=\lm(-\frac{3}{4}a_i+\frac{3}{8}(a_{i-j}+a_{i+j})+\frac{3}{2}\sg_j)=
-\frac{3}{4}+\frac{3}{8}+\frac{3}{8}=0=1\cdot 0=\lm(a_i)\lm(\sg_j)$. Finally, if $u=\sg_i$ and 
$v=\sg_j$ then $\lm(\sg_i\sg_j))=\lm(\frac{3}{4}(\sg_i+\sg_j)-\frac{3}{8}(\sg_{|i-j|}+
\sg_{i+j}))=0=\lm(\sg_i)\lm(\sg_j)$. So the equality holds in all cases and so $\lm$ is indeed 
an algebra homomorphism.
\end{proof}

Such a homomorphism is usually called a weight function and its existence shows that $HW$ is a 
baric (or weighted) algebra. Since $\lm$ is a homomorphism, its kernel $J$ is an ideal of 
codimension $1$.

Using $\lm$, we can also define a bilinear form on $HW$. Namely, for $u,v\in HW$, we set 
$(u,v):=\lm(u)\lm(v)$. It is immediate that this is a bilinear form; furthermore, it associates 
with the algebra product. Indeed, for $u,v,w\in HW$, we have $(uv,w)=\lm(uv)\lm(w)=
\lm(u)\lm(v)\lm(w)=\lm(u)\lm(vw)=(u,vw)$. In the theory of axial algebras such forms are called 
Frobenius forms. The form $(\cdot,\cdot)$ further satisfies the property that $(a_i,a_i)=
\lm(a_i)\lm(a_i)=1\cdot 1=1$, which is often required in the definition of a Frobenius form.

The next observation to make is that $HW$ is quite symmetric. Let $D$ be the infinite dihedral 
group acting naturally on $\Z$. For $\rho\in D$, let $\phi_\rho$ be the linear map that fixes 
all $\sg_j$ and sends $a_i$ to $a_{i^\rho}$.  Then $\phi_\rho$ is an automorphism of $HW$ and 
the map $\rho\mapsto \phi_\rho$ defines a faithful representation of $D$ as an automorphism 
group of $HW$. When $\chr(\F)\neq 3$, we claim that $Aut(HW)\cong D$. For this, we need the 
following observation.

\begin{lemma}\label{idempotents}
If $\chr(\F)\neq 3$, the elements $\{a_i\:|\:i\in \Z\}$ are the only nontrivial idempotents in 
$HW$.
\end{lemma}

\begin{proof}
Let 
$$e=\sum_{i\in\Z}r_ia_i+\sum_{j\in \Z_+}s_j\sg_j$$ 
be a nontrivial idempotent of $HW$. Suppose by contradiction that $e$ is not one of the $a_i$'s. 
First, suppose that $e$ involves both nonzero terms $a_i$ and nonzero terms $\sg_j$. Select 
the maximum $i$ with $r_i\neq 0$ and, similarly, the maximum $j$ with $s_j\neq 0$. Then, by 
(HW2), $e=e^2$ involves $a_{i+j}$, which is a contradiction. Thus, $e$ cannot contain both 
$a_i$'s and $\sg_j$'s. If $e$ contains no $a_i$'s, select the maximum $j$ with $s_j\neq 0$. Then 
it follows from (HW3) that $e=e^2$ involves $\sg_{2j}$, which is again a contradiction. It 
remains to consider the case where $e$ only involves $a_i$'s. Clearly, it must involve two 
different $a_i$. Say, let $i$ be maximum such that $r_i\neq 0$ and $i'$ be minimum such that 
$r_{i'}\neq 0$. Then $i\neq i'$ and (HW1) yields that $e=e^2$ involves $\sg_{i-i'}$; a 
contradiction.
\end{proof}

\begin{proposition}\label{aut}
If $\chr(\F)\neq 3$, then $Aut(HW)\cong D$. 
\end{proposition}

\begin{proof}
Let $\varphi\in\Aut(HW)$. By Lemma~\ref{idempotents}, $\varphi$ induces a permutation on the set 
$\{a_i\mid i\in\Z\}$ and consequently, by $(HW1)$, $\varphi$ induces a permutation on the set 
$\{\sg_j\mid j\in \Z_+\}$. Now observe that the action on the latter set has to be trivial. 
Indeed, for $\sg\in\{\sg_j\mid j\in\Z_+\}$, define the graph $\Gm_\sg$ with vertices $\{a_i,\mid 
i\in\Z\}$, where $a_i$ is adjacent to $a_k$ if and only if $\sg=a_ia_k-\frac{1}{2}(a_i+a_k)$. It 
is easy to see that if $\sg=\sg_j$ then $\Gm_\sg$ has exactly $j$ connected components. On the 
other hand, if $\sg^\varphi=\sg'$ then $\Gm_\sg^\varphi=\Gm_{\sg'}$. Since no two graphs 
$\Gm_\sg$ are isomorphic, we conclude that indeed $\varphi$ fixes all $\sg_j$.

In particular, the entire group $\Aut(HW)$ fixes $\sg_1$, and so it acts on the infinite string 
graph $\Gm_{\sg_1}$. Since this action is faithful, we conclude that $\Aut(HW)\cong D$.
\end{proof}

We are aiming to show that the $a_i$ are axes satisfying the fusion law 
$\M(2,\frac{1}{2})$. Since $D$ is transitive on the $a_i$'s, it suffices to check this for just 
one of them, say $a=a_0$. We start with the eigenvalues and eigenspaces of $\ad_a$.

Select $j\in\Z_+$ and set $U=\langle a,a_{-j},a_j,\sg_j\rangle$. It is immediate to see that $U$ 
is invariant under $\ad_a$, the latter being represented by the following matrix:
$$
\left(
\begin{array}{rrrr}
1&0&0&0\\
\frac{1}{2}&\frac{1}{2}&0&1\\
\frac{1}{2}&0&\frac{1}{2}&1\\
-\frac{3}{4}&\frac{3}{8}&\frac{3}{8}&\frac{3}{2}
\end{array}
\right)
$$
This has characteristic polynomial $x^4-\frac{7}{2}x^3+\frac{7}{2}x^2-x=
(x-1)x(x-2)(x-\frac{1}{2})$ and eigenspaces $U_1=\langle a\rangle$, $U_0=\langle 
u_j\rangle$, $U_2=\langle v_j\rangle$, $U_{\frac{1}{2}}=\langle w_j\rangle$, where
\begin{eqnarray} \label{uvw}
&&u_j:=6a-3(a_{-j}+a_j)+4\sg_j \nonumber\\ 
&&v_j:= 2a-(a_{-j}+a_j)-4\sg_j \\ 
&&w_j:=a_{-j}-a_j. \nonumber
\end{eqnarray}
The only exception to this statement arises when $\chr(\F)=3$. In this case, $2=\frac{1}{2}$ 
and so $U_2$ and $U_{\frac{1}{2}}$ merge into a single $2$-dimensional eigenspace 
$U_2=U_{\frac{1}{2}}=\langle v_j,w_j\rangle$.

In all cases, we can write that $U=\langle a,u_j,v_j,w_j\rangle$, that is, $a_{-j}$, $a_j$ 
and $\sg_j$ can be expressed via these vectors. From this we deduce the following.

\begin{lemma}
The adjoint map $\ad_a$ is semisimple on $HW$ and
\begin{enumerate}
\item[\rm(a)] if $\chr(\F)\neq 3$ then the spectrum of $\ad_a$ is $\{1,0,2,\frac{1}{2}\}$ and 
the eigenspaces are $HW_1=\langle a\rangle$, $HW_0=\langle u_j\mid j\in\Z_+\rangle$, 
$HW_2=\langle v_j\mid j\in\Z_+\rangle$, and $HW_{\frac{1}{2}}=\langle w_j\mid j\in \Z_+\rangle$;
\item[\rm(b)] if $\chr(\F)=3$ then the spectrum is $\{1,0,\frac{1}{2}\}$ and the eigenspaces are 
$HW_1=\langle a\rangle$, $HW_0=\langle u_j\mid j\in\Z_+\rangle$, and $HW_{\frac{1}{2}}=\langle 
v_j,w_j\mid j\in \Z_+\rangle$.
\end{enumerate}
\end{lemma}

In order to avoid the complication arising in characteristic $3$, we will use the notation 
$HW_u:=\langle u_j\mid j\in\Z_+\rangle$, $HW_v:=\langle v_j\mid j\in\Z_+\rangle$, and 
$HW_w=\langle w_j\mid j\in \Z_+\rangle$ calling these subspaces the $u$-, $v$-, and $w$-parts 
of $HW$, respectively. A similar terminology will be used for sums of these subspaces. Thus, 
in all characteristics we have the decomposition
$$HW=\langle a\rangle\oplus HW_u\oplus HW_v\oplus HW_w.$$
Let us relate this decomposition to the ideal $J$.

\begin{lemma}
$J=HW_u\oplus HW_v\oplus HW_w$.
\end{lemma}

\begin{proof}
By inspection, the weight function is zero on $HW_u\oplus HW_v\oplus HW_w$, so this sum is 
contained in $J$. Now the equality is forced because both $J$ and $HW_u\oplus HW_v\oplus HW_w$ 
have codimension $1$ in $HW$.
\end{proof}

So $J$ is the $uvw$-part of $HW$.

The stabiliser $D_a$ of $a$ in $D$ has order $2$ and it is generated by the involution $\tau$ sending every $a_i$ to 
$a_{-i}$ (and fixing every $\sg_j$). From this, it is immediate to see that the following holds.

\begin{lemma}
The involution $\tau$ acts as identity on $\langle a\rangle\oplus HW_u\oplus HW_v$ and as 
minus identity on $HW_w$. In particular, $HW_\tau=\langle a\rangle\oplus HW_u\oplus HW_v$ is 
the fixed subalgebra of $\tau$.
\end{lemma}

We find a further subalgebra by intersecting $HW_\tau$ with the ideal $J$. Namely, we set 
$V:=HW_\tau\cap J=HW_u\oplus HW_v$, the $uv$-part of $HW$. Clearly, $V$ is an ideal of 
$HW_\tau$. 

The vectors $u_j$ and $v_j$ form a basis of $V$. We will also use a second basis. For 
$j\in\Z_+$, let $c_j:=2a-(a_{-j}+a_j)$. These we combine with the elements $\sg_j$.

Let us, first of all, record the following, see the definition of $u_j$ and $v_j$ in 
(\ref{uvw}). 

\begin{lemma} \label{transition}
For all $j\in\Z_+$, we have $u_j=3c_j+4\sg_j$ and $v_j=c_j-4\sg_j$.
\end{lemma}

It is also easy to express $c_j$ and $\sg_j$ via $u_j$ and $v_j$. This means that the set of all 
vectors $c_j$ and $\sg_j$ is a basis of $V$. In the following lemma we compute the products for 
this basis. Note that we only need to compute the products $c_ic_j$ and $c_i\sg_j$, because 
$\sg_i\sg_j$ are given in (HW3). It will be convenient to use the following notation: for 
$i,j\in\Z_+$, we let $c_{i,j}:=-2c_i-2c_j+c_{|i-j|}+c_{i+j}$ and $\sg_{i,j}:= 
-2\sg_i-2\sg_j+\sg_{|i-j|}+\sg_{i+j}$. For example, (HW3) can now be restated as: 
$\sg_i\sg_j=-\frac{3}{8}\sg_{i,j}$.

\begin{lemma} \label{products c sigma}
For $i,j\in\Z_+$, we have that $c_ic_j=2\sg_{i,j}$ and $c_i\sg_j=\frac{3}{8}c_{i,j}$.
\end{lemma}

\begin{proof}
Using (HW1), $c_ic_j=(2a-(a_{-i}+a_i))(2a-(a_{-j}+a_j))=4a-2(\frac{1}{2}a_{-i}+\frac{1}{2}a+
\sg_i+\frac{1}{2}a_i+\frac{1}{2}a+\sg_i)-2(\frac{1}{2}a+\frac{1}{2}a_{-j}+\sg_j+\frac{1}{2}a+
\frac{1}{2}a_j+\sg_j)+(\frac{1}{2}a_{-i}+\frac{1}{2}a_{-j}+\sg_{|i-j|}+\frac{1}{2}a_i+
\frac{1}{2}a_{-j}+\sg_{i+j}+\frac{1}{2}a_{-i}+\frac{1}{2}a_j+\sg_{i+j}+\frac{1}{2}a_i+
\frac{1}{2}a_j+\sg_{|i-j|})$. Here all $a$, $a_{\pm i}$, $a_{\pm j}$ magically cancel yielding 
the first claim.

Similarly, using (HW2), $c_i\sg_j=(2a-(a_{-i}+a_i))\sg_j=2(-\frac{3}{4}a+\frac{3}{8}(a_{-j}+a_j)+
\frac{3}{2}\sg_j)-(-\frac{3}{4}a_{-i}+\frac{3}{8}(a_{-i-j}+a_{-i+j})+\frac{3}{2}\sg_j-
\frac{3}{4}a_i+\frac{3}{8}(a_{i-j}+a_{i+j})+\frac{3}{2}\sg_j)$. Here the $\sg_j$ cancel and the 
second claim follows after the terms are rearranged.
\end{proof}

Manifestly, all these products have a very uniform structure. Note that both $c_{i,j}$ and 
$\sg_{i,j}$ are symmetric in $i$ and $j$. In particular, it follows that $c_i\sg_j=\frac{3}{8}c_{i,j}=
\frac{3}{8}c_{j,i}=c_j\sg_i$. Another interesting consequence of this lemma is that the involution $\rho$ 
negating all $c_j$ and fixing all $\sg_j$ is an automorphism of $V$. We will however need 
another involution, and for this we need to compute products for the other basis, the vectors 
$u_j$ and $v_j$. 

Similarly to the above, we introduce $u_{i,j}:=-2u_i-2u_j+u_{|i-j|}+u_{i+j}$ and $v_{i,j}:=
-2v_i-2v_j+v_{|i-j|}+v_{i+j}$. Lemma \ref{transition} gives us the following.

\begin{lemma} \label{double transition}
For all $i,j\in\Z_+$, we have $u_{i,j}=3c_{i,j}+4\sg_{i,j}$ and $v_{i,j}=c_{i,j}-4\sg_{i,j}$.
\end{lemma}

Now we compute the products of the vectors $u_j$ and $v_j$.

\begin{lemma} \label{products u v}
For all $i,j\in\Z_+$, we have that $u_iu_j=3u_{i,j}$, $u_iv_j=-3v_{i,j}$, and $v_iv_j=-u_{i,j}$.
\end{lemma}

\begin{proof}
By Lemma \ref{transition}, $u_iu_j=(3c_i+4\sg_i)(3c_j+4\sg_j)=9c_ic_i+12c_i\sg_j+12\sg_ic_j+
16\sg_i\sg_j$. Using Lemma \ref{products c sigma}, this is $9(2\sigma_{i,j})+
12(\frac{3}{8}c_{i,j})+12(\frac{3}{8}c_{i,j})+16(-\frac{3}{8}\sg_{i,j})=9c_{i,j}+12\sg_{i,j}=
3u_{i,j}$. For the last step, we used Lemma \ref{double transition}.

Similarly, $u_iv_j=(3c_i+4\sg_i)(c_j-4\sg_j)=3c_ic_j+4\sg_ic_j-12c_j\sg_j-16\sg_i\sg_j=
3(2\sg_{i,j})+4(\frac{3}{8}c_{i,j})-12(\frac{3}{8}c_{i,j})-16(-\frac{3}{8}\sg_{i,j})=
-3c_{i,j}+12\sg_{i,j}=-3v_{i,j}$ and $v_iv_j=(c_i-4\sg_i)(c_j-4\sg_j)=c_ic_j-4\sg_ic_j-4c_i\sg_j+ 
16\sg_i\sg_j=2\sg_{i,j}-4(\frac{3}{8}c_{i,j})-4(\frac{3}{8}c_{i,j})+16(-\frac{3}{8}\sg_{i,j})=
-3c_{i,j}-4\sg_{i,j}=-u_{i,j}$.
\end{proof}

Just like Lemma \ref{products c sigma} implied that $V$ admits an involution $\rho$, this lemma 
gives us that the involution $\theta$ fixing every $u_j$ and negating every $v_j$ is an 
automorphism of $V$. 

We are now ready to establish the fusion law for the axis $a$.

\bigskip\noindent
{\it Proof of Theorem~\ref{HW}.}
Recall that, when $\chr(\F)\neq 3$, we have that $\langle a\rangle=HW_1$, $HW_u=HW_0$, 
$HW_v=HW_2$, and $HW_w=HW_{\frac{1}{2}}$ are the eigenspaces of $\ad_a$. So we need to see how 
these subspaces behave under the product in $HW$.

First of all, clearly, $\langle a\rangle HW_u=0$, $\langle a\rangle HW_v=HW_v$ and $\langle 
a\rangle HW_w=HW_w$. Now we focus on the products of the $u$-, $v$-, and $w$-parts, which are 
all contained in $J$, and so their products are also contained in $J$, avoiding $\langle 
a\rangle$. Furthermore, the involution $\tau$ induces a grading on $J$, under which 
$V=HW_u\oplus HW_v$ is the even part and $H_w$ is the odd part. Hence $HW_uHW_w$ and $HW_vHW_w$ 
are contained in $HW_w$, while $HW_wHW_w\subseteq V=HW_u\oplus HW_v$. It remains to check the 
products within $V$. Similarly, the involution $\theta$ induces a grading on $V$, under which 
$HW_u$ is even and $HW_v$ is odd. Hence $HW_uHW_u\subseteq HW_u$, $HW_uHW_v\subseteq HW_v$, and 
$HW_vHW_v\subseteq HW_u$. (Of course, these three inclusions can be seen directly from Lemma 
\ref{products u v}.) Hence, when $\chr(\F)\neq 3$, $a$ satisfies the fusion law of 
Table~\ref{HWtable}, which is slightly stricter than the fusion law $\M(2,\frac{1}{2})$. 
\begin{table}
$$ 
\begin{array}{|c||c|c|c|c|}
\hline
\star & 1 & 0 & 2 & \frac{1}{2}\\
\hline
\hline
1 & 1 & & 2 & \frac{1}{2}\\
\hline
0 & & 0 & 2 & \frac{1}{2}\\
\hline
2 & 2 & 2 & 0 & \frac{1}{2}\\
\hline
\frac{1}{2} & \frac{1}{2} & \frac{1}{2} & \frac{1}{2} & 0,2\\
\hline
\end{array}
$$
\caption{Fusion law for $HW$}\label{HWtable}
\end{table} 

Since $D$ is transitive on the $a_i$, all of them are axes satisfying the law in Table 
\ref{HWtable}. To complete the proof of Theorem \ref{HW}, it remains to show that 
$HW=\langle\!\langle a_0,a_1\rangle\!\rangle$. Let $H:=\langle\!\langle 
a_0,a_1\rangle\!\rangle$. Note that $\sg_1=a_0a_1-\frac{1}{2}(a_0+a_1)\in H$. Also, 
$\frac{3}{8}a_{-1}=a_0\sg_1+\frac{3}{4}a_0-\frac{3}{8}a_1-\frac{3}{2}\sg_1\in H$. Assuming that 
$\chr(\F)\neq 3$, this gives us that $a_{-1}\in H$. Clearly, $=\langle\!\langle 
a_0,a_1\rangle\!\rangle$ is invariant under the involution $\pi\in D$ switching $a_0$ and $a_1$. 
Now we also see that $H=\langle\!\langle a_{-1},a_0,a_1\rangle\!\rangle$ is invariant under the 
involution $\tau\in D_a$. Since $D=\langle\tau,\pi\rangle$, this makes $H$ invariant under all 
of $D$, and so $H$ contains all axes $a_i$. Clearly, this means that $H=HW$. This completes the 
proof of Theorem \ref{HW}.\hfill$\square$

\bigskip
Let us discuss what happens when $\chr(\F)=3$. The fusion that we established for the 
decomposition $HW=\langle a\rangle\oplus HW_u\oplus HW_v\oplus HW_w$ remains true, but it 
becomes even stricter because so many coefficients are multiples of $3$. For example, 
$HW_uHW_u=0=HW_uHW_v$. On the other hand, $2=\frac{1}{2}$ in characteristic $3$, so we need to 
merge $HW_v$ and $HW_w$ into a single eigenspace $HW_{\frac{1}{2}}=HW_v\oplus HW_w$. Let us see 
what fusion we get inside this eigenspace. We already know that $HW_vHW_v\subseteq HW_u$.

\begin{lemma}
If $\chr(\F)=3$ then $HW_vHW_w=0$ and $HW_wHW_w\subseteq HW_u$.
\end{lemma}

\begin{proof}
Indeed, for $i,j\in\Z_+$, $v_iw_j=(2a-(a_{-i}+a_i)-4\sg_i)(a_{-j}-a_j)$. Taking into account 
that $aw_j=\frac{1}{2}w_j$ and that $\sg_iw_h=0$ when $\chr(\F)=3$ (see (HW2)), we obtain that 
$v_iw_j=(a_{-j}-a_j)-(a_{-i}+a_i)(a_{-j}-a_j)=(a_{-j}-a_j)-(\frac{1}{2}(a_{-i}+a_{-j})+
\sg_{|i-j|})+(\frac{1}{2}(a_i+a_{-j})+\sg_{i+j})-(\frac{1}{2}(a_{-i}+a_j)+\sg_{i+j})+
(\frac{1}{2}(a_i+a_j)+\sg_{|i-j|})=0$, because everything cancels here.

Also, $w_iw_j=(a_{-i}-a_i)(a_{-j}-a_j)=(\frac{1}{2}(a_{-i}+a_{-j})+\sg_{|i-j|})-
(\frac{1}{2}(a_i+a_{-j})+\sg_{i+j})-(\frac{1}{2}(a_{-i}+a_j)+\sg_{i+j})+
(\frac{1}{2}(a_i+a_j)+\sg_{|i-j|})=2\sg_{|i-j|}-2\sg_{i+j}=\frac{1}{2}u_{|i-j|}-
\frac{1}{2}u_{i+j}$. The last equality holds because $u_k=4\sg_k$ when $\chr(\F)=3$ (see 
(\ref{uvw})).
\end{proof}

So we can see now that in characteristic $3$ the axes $a_i$ in $HW$ satisfy the fusion law in 
Table \ref{HWtable3}.
\begin{table}
$$ 
\begin{array}{|c||c|c|c|}
\hline
\star & 1 & 0 & \frac{1}{2}\\
\hline
\hline
1 & 1 & & \frac{1}{2}\\
\hline
0 & & & \frac{1}{2}\\
\hline
\frac{1}{2} & \frac{1}{2} & \frac{1}{2} & 0\\
\hline
\end{array}
$$
\caption{Fusion law for $HW$ in characteristic $3$}\label{HWtable3}
\end{table} 
In particular, $HW$ is an algebra of Jordan type $\frac{1}{2}$. Let us prove that, in fact, 
$HW$ is a Jordan algebra when $\chr(\F)=3$. For this we need to verify the Jordan identity
$x(yx^2)=(xy)x^2$
for all $x,y\in HW$. For an element 
$$x=\sum_{i\in\Z}r_ia_i+\sum_{j\in \Z_+}s_j\sg_j\in HW,$$ 
we call $a(x):=\sum_{i\in\Z}r_ia_i$ the $a$-part of $x$ and we call $\sg(x):=\sum_{j\in 
\Z_+}s_j\sg_j$ the $\sg$-part of $x$. According to (HW2) and (HW3), when $\chr(\F)=3$, we have 
$\sg_jHW=0$ for all $j\in\Z_+$. In particular, we have that $xy=a(x)a(y)$.

\begin{lemma}
If $x,y\in HW$, 
with $x=a(x)$ and $y=a(y)$,
then $a(xy)=\frac{\lm(y)}{2}x+\frac{\lm(x)}{2}y$.
\end{lemma}

\begin{proof}
Indeed, let $x=\sum_{i\in\Z}r_ia_i$ and $y=\sum_{i\in\Z}t_ia_i$. Note that $\lm(x)=\sum_{i\in 
Z}r_i$ and $\lm(y)=\sum_{i\in\Z}t_i$. Now we have from (HW1) that
\begin{align*}
a(xy)&=\sum_{i,j\in\Z}r_it_j\frac{1}{2}(a_i+a_j)\\
     &=\sum_{i,j\in\Z}r_it_j\frac{1}{2}a_i+\sum_{i,j\in\Z}r_it_j\frac{1}{2}a_j\\
     &=\frac{1}{2}(\sum_{j\in\Z}t_j)(\sum_{i\in\Z}r_ia_i)+\frac{1}{2}(\sum_{i\in\Z}r_i)
     (\sum_{j\in\Z}t_ja_j)\\
     &=\frac{\lm(y)}{2}x+\frac{\lm(x)}{2}y.
\end{align*}
\end{proof}

In particular, this lemma implies that $a(x^2)=\lm(x)a(x)$ for all $x\in HW$.

\begin{theorem}
\label{Jord}
If $\chr(\F)=3$ then $HW$ is a Jordan algebra.
\end{theorem}

\begin{proof}
For $x,y\in HW$, we have  
$$x(yx^2)=a(x)(a(y)a(x^2))
=a(x)(a(y)\lm(x)a(x))=\lm(x)a(x)(a(y)a(x))=
$$
$$
(a(x)a(y))\lm(x)a(x)=
(a(x)a(y))a(x^2)=(xy)x^2.$$ Note that we used that 
$HW$ is commutative.
\end{proof}

Note also that with the same argument used for Theorem~\ref{Jord} one can prove the 
following more general result.

\begin{proposition} \label{Jord 2}
Every commutative baric algebra $B$, with weight function $\om$, such that
\begin{enumerate}
\item[\rm(a)] $B=A\oplus I$, where $A$ is a subspace and $I$ is an ideal,
\item[\rm(b)] $IB=0$, and
\item[\rm(c)] for all $a,b\in A$, $ab-\frac{1}{2}(\om(b)a+\om(a)b)\in I$, 
\end{enumerate}
is a Jordan algebra. 
\end{proposition}

Conversely, take arbitrary vector spaces $A$ and $I$ over $\F$ (of characteristic not $2$), 
an arbitrary linear map $\om:A\to\F$, and an arbitrary symmetric bilinear map $\sg:A\times 
A\to I$. Define the algebra product on $B=A\oplus I$ by $BI=0$ and $ab:=
\frac{1}{2}(\om(b)a+\om(a)b)+\sg(a,b)$ for $a,b\in A$. Then this algebra $B$ satisfies the 
assumptions of Proposition \ref{Jord 2} and hence $B$ is a baric Jordan algebra. Note that 
$\om$ will be the weight function of $B$ if we extend it to $B$ via $\om(I)=0$. (Note that 
this may not be the only choice for the weight function.)

Manifestly, $HW$ in characteristic $3$ can be obtained by this construction with $\om=\lm$.

Now that we have proved our main results, let us tie the loose ends. 

\begin{proposition}
\begin{itemize}
\item[\rm(a)] The form $(\cdot,\cdot)$ is, up to a scalar factor, the only Frobenius form on 
$HW$.
\item[\rm(b)] Every proper ideal of $HW$ is contained in $J$, which is the radical of $HW$.
\end{itemize}
\end{proposition}

\begin{proof}
We refer to \cite{KMS} for a detailed discussion of the radical of an axial algebra, projection 
form and projection graph.

The form $(\cdot,\cdot)$ that we introduced on $HW$ is the projection form, because it is a 
Frobenius form satisfying $(a_i,a_i)=1$ for all $i\in\Z$. The projection graph has all $a_i$ as 
vertices with an edge between $a_i$ and $a_j$, $i\neq j$, whenever $(a_i,a_j)\neq 0$. In fact, 
one can easily check that $(a_i,a_j)=1$ for all $i,j\in\Z$, so the projection graph of $HW$ is 
the complete graph. In particular, it is connected. It now follows from \cite[Proposition 
4.19]{KMS} that $HW$ has only one Frobenius form up to a scalar factor. The same connectivity 
property implies, by \cite[Corollary 4.15]{KMS}, that every proper ideal of $HW$ is contained 
in the radical of $HW$ (the largest ideal not containing axes). Finally, by \cite[Corollary 
4.11]{KMS}, the radical of $HW$ coincides with the radical of the projection form 
$(\cdot,\cdot)$, and that is $J$.
\end{proof}

Recall that we earlier introduced automorphisms $\rho$ and $\theta$ of the subalgebra 
$V=HW_u\oplus HW_v$. In the final result of the section we indicate the connection between these 
two involutions.

\begin{lemma}
The linear transformation $\psi:V\to V$ defined by:
$$c_j\mapsto\frac{1}{2}v_j=\frac{1}{2}c_j-2\sg_j\mbox{ and }
\sg_j\mapsto-\frac{1}{8}u_j=-\frac{3}{8}c_i-\frac{1}{2}\sg_j,$$
for $j\in\Z_+$, is an automorphism of $V$. Furthermore, $\rho^\psi=\theta$.
\end{lemma}

This follows from the formulae in Lemmas \ref{products c sigma} and \ref{products u v}. Note 
that $\psi$ is also an involution because the matrix 
$$
\left(\begin{array}{rr}
\frac{1}{2} & -2\\
-\frac{3}{8} & -\frac{1}{2}
\end{array}\right)
$$
squares to identity. This means that $\langle\rho,\psi\rangle\cong D_8$. Interestingly, the 
subalgebra $V$ appears to have symmetries independent of the entire algebra $HW$, because the 
automorphisms from $\langle\rho,\psi\rangle$ do not seem to extend to $HW$. 

\section{Discussion}

In this final section we pose some questions related to $HW$. Recall that it is known that every 
symmetric $2$-generated primitive algebra of Monster type $(\al,\bt)$ over a field of 
characteristic other that $2$ is at most $8$-dimensional unless $(\al,\bt)=(2,\frac{1}{2})$.

\begin{question} Is $HW$ the only infinite-dimensional primitive algebra of Monster type? Is it 
the only such algebra having dimension greater than $8$?
\end{question}

The answer to this question could depend on knowing all ideals of $HW$.

\begin{question}
Does $HW$ contain any nonzero ideals of finite dimension? Is it possible to classify all ideals 
in $HW$?
\end{question}

Note that $J$ is not the only proper nonzero ideal of $HW$. It can be shown that the ideal 
generated by all elements $\sg_j$ has codimension $1$ in $J$ (and hence codimension $2$ in 
$HW$). 

In addition to finding all ideals of $HW$, it is also interesting to find all of its subalgebras 
and their automorphisms.

\begin{question}
Which is the full automorphism group of $HW$ in characteristic $3$? Which is the automorphism 
group of $V$? Which automorphisms of $V$ extend to larger subalgebras of $HW$?
\end{question}

Finally, the fact that $HW$ is a baric algebra looks exciting. One can see easily that a baric 
algebra of Monster type satisfies the fusion law as in Table \ref{HWtable}, but with $2$ and 
$\frac{1}{2}$ substituted with arbitrary $\al$ and $\bt$.

\begin{question}
Are there examples of baric algebras of Monster type $(\al,\bt)$ for other values of $\al$ and 
$\bt$? Is it possible to classify all such algebras?
\end{question}

\end{document}